\documentclass[reqno,11pt]{amsart}
\usepackage{amsmath,amssymb,amsfonts,amsthm, amscd,indentfirst}
\usepackage{amsmath,latexsym,amssymb,amsmath,
	amscd,amsthm,amsxtra}
\usepackage{color}
\usepackage{pdfpages}
\usepackage{graphics}
\usepackage{subfig}
\usepackage{epstopdf}
\usepackage{epsfig}
\usepackage{enumitem}
\usepackage{float}
\usepackage[usenames,dvipsnames]{xcolor}

\usepackage{varioref}
\usepackage{hyperref}
\usepackage[noabbrev,capitalise,nameinlink]{cleveref}

\usepackage[T1]{fontenc}
\usepackage{newtxtext,newtxmath}

\hypersetup{colorlinks={true},linkcolor={blue},citecolor=blue, urlcolor={blue}}
\usepackage[english]{babel}
\usepackage{csquotes}

\usepackage{graphicx}
\usepackage{caption}

\newtheorem{theorem}{Theorem}[section]
\newtheorem{lemma}{Lemma}[section]
\newtheorem{proposition}{Proposition}[section]
\newtheorem{definition}{Definition}[section]
\newtheorem{remark}{Remark}[section]

\providecommand{\abs}[1]{\lvert#1\rvert}

\providecommand{\norm}[1]{\lVert#1\rVert}
\providecommand{\Norm}[1]{\left\lVert#1\right\rVert}

\def\De{\Delta}
\def\O{\Omega}
\def\S{\Sigma}

\def\p{\partial}
\def\n{\nabla}
\def\<{\langle}
\def\>{\rangle}

\newcommand{\mfR}{\mathbb{R}}
\newcommand{\mfS}{\mathbb{S}}

\newcommand{\rd}{{\rm d}}

\numberwithin{equation} {section}

\begin{document}
	
		\title[Rigidity and stability]{Rigidity and quantitative stability for partially overdetermined problems and capillary CMC hypersurfaces}
	\author[Jia]{Xiaohan Jia}
	\address[X.J]{School of Mathematics\\
Southeast University\\
 211189, Nanjing, P.R. China}
	\email{jiaxiaohan@xmu.edu.cn}
	\author[Lu]{Zheng Lu}
	\address[Z.L]{School of Mathematical Sciences\\
		Xiamen University\\
		361005, Xiamen, P.R. China
  \newline\indent Mathematisches Institut \\
  Albert-Ludwigs-Universität Freiburg\\
  Freiburg im Breisgau, 79104, Germany}
	\email{zhenglu@stu.xmu.edu.cn}
 \author[Xia]{Chao Xia}
	\address[C.X]{School of Mathematical Sciences\\
		Xiamen University\\
		361005, Xiamen, P.R. China}
	\email{chaoxia@xmu.edu.cn}
	\author[Zhang]{Xuwen Zhang}
	\address[X.Z]{School of Mathematical Sciences\\
		Xiamen University\\
		361005, Xiamen, P.R. China
\newline\indent Institut f\"ur Mathematik, Goethe-Universit\"at, 
60325, Frankfurt, Germany}
	\email{zhang@math.uni-frankfurt.de}
	\thanks{This work is  supported by the NSFC (Grant No. 12271449, 12126102)}
	
		\begin{abstract}
	In this paper, we first prove a rigidity result for a Serrin-type partially overdetermined problem in the half-space, which gives a characterization of capillary spherical caps by the overdetermined problem. In the second part, we prove
    quantitative stability results for the Serrin-type partially overdetermined problem, as well as capillary almost constant mean curvature hypersurfaces in the half-space.
		
		\
		
  \\
		
	\end{abstract}
	
	\maketitle
	
	\medskip
	
\section{Introduction}\label{Sec-1}
The well-known Alexandrov's soap bubble theorem \cite{Aleksandrov62} says that an embedded closed constant mean curvature (CMC) hypersurface in $\subset\mfR^{n+1}$ must be a round sphere. Serrin \cite{Serrin71} proposed the overdetermined boundary value problem in a bounded domain $\O\subset\mfR^{n+1}$:
\begin{align*}
\begin{cases}
\bar\De f=1\quad&\text{in }\O,\\
f=0,\quad&\text{on }\p\O,\\
\bar\nabla_\nu f=c_0,\quad&\text{on }\p\O,
\end{cases}
\end{align*} for some $c_0\in \mathbb{R}$ and he proved that it
admits a solution if and only if $\p\O$ is a round sphere. Both results use Alexandrov's moving plane method. On the other hand, Reilly \cite{Reilly77}, Ros \cite{Ros87} and Weinberger \cite{Weinberger71} reproved the above results by integral method. 

Motivated by the study of equilibrium shape of liquid drops and crystals in a given solid container, we are interested in the so-called capillary hypersurfaces,
which are critical points of the free energy functional under volume constraint.
For the history of the study of capillary phenomenon, we refer to the monograph \cite{Finn86}.

Alexandrov-type theorem for capillary CMC hypersurfaces in the upper half-space has been proved by Wente \cite{Wente80} via the moving plane method.  Reilly-Ros-type proof has been given in \cite{JXZ22} (see also \cite{JWXZ22-B, JWXZ22-A}). It is natural to ask a Serrin-type problem in the half-space, which characterizes capillary spherical cap. This is the first purpose of this paper.

Let $\mfR^{n+1}_+=\{x\in\mfR^{n+1}:x_{n+1}>0\}$ be  the upper half-space and $E_{n+1}$ be the $(n+1)$-coordinate unit vector. 
Let $\O\subset\mfR^{n+1}_+$ be a bounded domain with piece-wise smooth boundary $\p\O=\S\cup T$, where $\S\coloneqq\overline{\p\O\cap \mfR^{n+1}_+}$ is a compact smooth hypersurface in $\overline{ \mfR^{n+1}_+}$, $T\coloneqq\p\Omega\cap\p \mfR^{n+1}_+$ and $\S$ meets $T$ at an $(n-1)$-dimensional submanifold $\Gamma\coloneqq\S\cap T$. Throughout this paper, $\O$ will be in this form.
In the following, we denote by $B_r(z)$ the ball of radius $r$ centered at $z$.
We consider the following overdetermined mixed boundary value problem in $\O$:
\begin{equation}\label{eq-mixed-halfspace0}
\begin{cases}
  \bar\Delta f=1 &\text{in } \Omega\\
  f=0 &\text{on }\Sigma\\
  \bar\nabla_{\bar N} f={ c}
  &\text{on } T,
\end{cases}
\end{equation}
where $\bar N =-E_{n+1}$ is the downward unit normal to $\p\mfR_+^{n+1}$ and $c\in\mfR$.




We first prove the following Serrin-type theorem in the half-space.
\begin{theorem}\label{Thm-rigid-BVP-halfspace}
Let $\O\subset\mfR^{n+1}_+$ be as above.
If the mixed boundary value problem \eqref{eq-mixed-halfspace0} together with $\bar \nabla_\nu f\equiv c_0$ on $\S$ admits a solution $f\in W^{1,\infty}(\O)\cap W^{2,2}(\O)$ such that $f\le 0$,
then $c_0>0$ and $\O$ must be of the form
\begin{align*}
    \O_{n+1,c,c_0}=B_{(n+1)c}(z) \cap\mfR^{n+1}_+,
\end{align*}
where
$z=z'+(n+1)cE_{n+1}$ for some $z'\in\p\mfR^{n+1}_+$,
and the solution $f$ is a quadratic function given by
\begin{align*}
    f(x)=\frac{\vert x-z\vert^2-(n+1)^2c_0^2}{2(n+1)}.
\end{align*}
In particular, $\S$ is a spherical cap that meets $\p\mfR^{n+1}_+$ at a constant contact angle $\theta$, characterized by
$ \cos\theta=-\frac{c}{c_0}.$
\end{theorem}
\begin{remark}
\normalfont
\begin{itemize}

    \item[(i)] The case for $c=0$ can be easily reduced to Serrin's classical result by using a hyperplane reflection.
    \item[(ii)] Note that when $c\le 0$, by maximum principle,
    the condition $f\le 0$ is automatically satisfied. Observe that $\theta\in(0,\pi/2]$ when $c\le 0$.
    \item[(iii)] If $\O$ is such that the contact angle at each point that $\S$ meets $T$ is $\le \frac{\pi}{2}$, then the condition $f\in W^{1,\infty}(\O)\cap W^{2,2}(\O)$ is automatically satisfied by Lieberman's regularity result \cite{Liebermann89}, see also \cite[Appendix A]{JXZ22}.
 \end{itemize}
\end{remark}

To prove \cref{Thm-rigid-BVP-halfspace}, we establish the following crucial integral identity:
\begin{align}\label{int-iden1}
 &\int_{\Omega}(-f)\left\{\vert\bar\nabla^2f\vert^2-\frac{(\bar\De f)^2}{n+1}\right\}\rd x\notag\\
 =&\frac12\int_\Sigma\left(f_\nu^2-\left(\frac{R(c,\O)}{n+1}\right)^2\right)\bigg(f_{\nu}-\Big\<\frac{x+(n+1)cE_{n+1}}{n+1},\nu\Big\>\bigg) \rd A.
 \end{align}
Here
\begin{align*}
    R(c,\O)\coloneqq(n+1)\frac{\vert\Omega\vert-c\vert T\vert}{\vert\S\vert},
\end{align*}
 which will be referred to as the reference radius of $\O$. Here and in the following  we abbreviate $f_\nu=\bar \nabla_\nu f$ for notation simplicity. We remark that such integral identity has been first derived by Magnanini-Poggesi \cite{MP19} for Serrin's classical problem.

\


In the second part of the paper, we study quantitative stability of Alexandrov-type theorem and Serrin-type theorem. The problems for the case of closed hypersurfaces have been studied extensively in the last decades, see for example \cite{DM05,DM06,BNST08,CM17,DMMN18,DG19,DG21,MP19, MP20-1,MP20-2, Scheuer21,SX22}. In particular, in a series of papers, Magnanini-Poggesi \cite{MP19,MP20-1,MP20-2} revisited the integral proofs of Reilly-Ros and Weinberger, and established two kinds of integral identities, by virtue of which they studied quantitative stability for Alexandrov's theorem and Serrin's theorem. The idea of  Magananini-Poggesi is to consider the difference of the exterior radius and the interior radius of a closed hypersurface $\S$, and intuitively, the difference measures the proximity of $\S$ to a round sphere. The key point of their argument is that by virtue of the integral identities, one can get a quantitative control on the difference.
In a most recent work \cite{MP22}, they study the quantitative stability of the partially overdetermined problem in a half-ball proposed by Guo-Xia \cite{GX19}, where rigidity is achieved by free boundary spherical caps.
Our results serve as the extension from the above results to the hypersurfaces with boundary in  the half-space. The rigidity and quantitative stability problem for hypersurfaces with boundary in the half-ball case will be studied in a separate forthcoming paper.
We premise the following definitions to introduce our main quantitative stability results.
\begin{definition}[exterior radius and interior radius, \cite{MP22}]
    \normalfont
Let $\O\subset\mfR_+^{n+1}$ and $\S=\overline{\p\O\cap \mfR_+^{n+1}}$. For a properly chosen center $O\subset\mfR^{n+1}$, the \textit{exterior radius} and the \textit{interior radius} (with respect to $O$) are defined respectively as $$\rho_e\coloneqq\rho_e(\O,O)\coloneqq\max_{x\in\S}|x-O|,\quad \rho_i\coloneqq\rho_i(\O,O)\coloneqq\min_{x\in\S}|x-O|.$$
\end{definition}

\begin{definition}[{$(\eta,a)$-uniform interior cone condition}, \cite{MP22}]
    \normalfont
    Given $\eta\in(0,\pi/2]$ and $a>0$.  $\O\subset \mathbb{R}_+^{n+1}$ is said to satisfy the $(\eta,a)$-uniform interior cone condition, if for every $x\in\p\O$, there exists a unit vector $\omega_x$ such that the cone with vertex at the origin, axis $\omega_x$, opening angle $\eta$, and height $a$, defined as
    \begin{align*}
        C_{\omega_x,a}
        =\{y:\left<y,\omega_x\right>>\vert y\vert\cos\eta, \vert y\vert<a\}
    \end{align*}
    is such that
    \begin{align*}
        v+C_{\omega_x,a}\subset \O,\quad\forall v\in B_a(x)\cap\overline{\O}.
    \end{align*}
\end{definition}
The uniform interior cone condition is equivalent to the Lipschitz regularity of the domain $\O$, see \cite[Section 3]{MP22} and the references therein for a detailed account.



Next we introduce a uniform exterior $\theta$-spherical cap condition, which will be used for deriving gradient estimate for the solutions to \eqref{eq-mixed-halfspace0}. It could be regarded as the capillary variant of the classical uniform exterior sphere condition. Such condition is motivated by \cite[Definition 4.6]{Pog22}, where a relevant uniform exterior sphere condition was proposed for domains in a cone.  
\begin{definition}[{$r_e$-uniform exterior $\theta$-spherical cap condition}]\label{Defn-exterior-sphere}
    \normalfont
Let $\O\subset\mfR^{n+1}_+$  and $\S=\overline{\p\Omega\cap\mfR^{n+1}_+}$.
Given $\theta\in(0,\pi)$ and $r_e>0$,
 \textit{$\Omega$ is said to satisfy the $r_e$-uniform exterior $\theta$-spherical cap condition},
if for any $x\in\Sigma$, there exists a touching ball  $B_{r_e}(O_x)$ of radius $r_e$ and center $O_x$ such that
\begin{enumerate}
        \item $\overline{ B_{r_e}(O_x)}\cap \overline{\Omega}=\{x\}$,
        \item $\left<O_x,E_{n+1}\right>\geq r_e\cos\theta$.
\end{enumerate}
\end{definition}
\begin{remark}
    \normalfont
   \begin{itemize}
   
   \item[(i)] Condition (1) in \cref{Defn-exterior-sphere} simply means that the sphere we find touches $\S$ from outside. Condition  (2) is the key feature of our definition, since we propose a height control of the touching sphere. 
   \item[(ii)] When $\theta=\frac{\pi}{2}$, the condition reduces to the one defined by Poggesi \cite[Definition 4.6]{Pog22}.
   \item[(iii)] Let $\O\subset\mfR^{n+1}_+$   be such that the contact angle $\tilde\theta(x)$ at  each $x\in \Gamma$  satisfies
    $0<\tilde\theta(x)\leq \theta$ for some $\theta\in(0,\pi)$. Then $\O$ satisfies the $r_e$-uniform exterior $\theta$-spherical cap condition for some $r_e>0$. See \cref{nat-exterior-sphere}.
    \end{itemize}
\end{remark}

We shall use $d_\O$ to denote the diameter of $\O$ and $\|\cdot\|_{1,\S}$ to denote $L^1$-norm on $\S$. Our second result is the following quantitative stability for Serrin-type result (\cref{Thm-rigid-BVP-halfspace}).
\begin{theorem}\label{Thm-stabi-OverDetermined-halfspace}
Let $\O\subset\mathbb{R}_+^{n+1}$ satisfy the $(\eta,a)$-uniform interior cone condition and the $r_e$-uniform exterior $\theta$-spherical cap condition for some $\eta\in(0,\pi/2]$, $a>0$, $\theta\in(0,\pi/2], r_e>0$.
Let $f\in W^{1,\infty}(\O)\cap W^{2,2}(\O)$ be a solution to \eqref{eq-mixed-halfspace0} with some $c\leq 0$. Let $O$ be chosen as
\begin{align}\label{defn-O-halfspace}
  \frac{1}{\abs{\O}}\bigg(\int_{\O}x\rd x-(n+1)\int_{T}f\rd A\cdot E_{n+1}\bigg),
\end{align}
and $R(c,\O)$ be given by \eqref{defn-ref-radius1}.

Then there exists some constant $C=C(n,c,d_\O,\eta,a, \theta,r_e)>0$ such that
\begin{align*}
    \rho_e-\rho_i
    \leq C
\begin{cases}
\Norm{f_\nu^2-\left(\frac{R(c,\O)}{n+1}\right)^2}^{\frac{1}{n+1}}_{1,\S},\quad&n\geq2,\\
\Norm{f_\nu^2-\left(\frac{R(c,\O)}{n+1}\right)^2}^{\frac{1}{2}}_{1,\S}\max\left\{\log\left(\Norm{f_\nu^2-\left(\frac{R(c,\O)}{n+1}\right)^2}^{-\frac{1}{2}}_{1,\S}\right),1\right\},
\quad&n=1.
\end{cases}
\end{align*}
\end{theorem}
Recall that in \cite{JXZ22}, we utilized the solution to \cref{eq-mixed-halfspace0} in Reilly's formula to reprove Wente's Alexandrov-type theorem for capillary CMC hypersurfaces in the half-space. Indeed, applying the Reilly's formula on the solution to \cref{eq-mixed-halfspace0} yields the following integral identity: 
\begin{align}\label{iden-integral-meancurvature-2}
&\int_\Omega\left\{\vert\bar\nabla^2f\vert^2-\frac{(\bar\De f)^2}{n+1}\right\}\rd x+\frac{n}{R(\theta,\O)}\int_\S\left(f_\nu-\frac{R(\theta,\O)}{n+1}\right)^2\rd A\notag\\
=&\int_\S(H-H(\theta,\O))f_\nu^2\rd A,
\end{align} 
where $R(\theta,\O)$ and $H(\theta,\O)$ are the $\theta$-reference radius and $\theta$-reference mean curvature given by \eqref{defn-ref-radius11} and \eqref{defn-ref-meancurvature1} respectively.
Relying on this identity and the similar analysis as in the proof of \cref{Thm-stabi-OverDetermined-halfspace}, we obtain the following quantitative stability for this Alexandrov-type theorem.
\begin{theorem}\label{Thm-stabi-Meancurvature-halfspace}
Let $\O\subset\mathbb{R}_+^{n+1}$ satisfy the $(\eta,a)$-uniform interior cone condition and the $r_e$-uniform exterior $\theta$-spherical cap condition for some $\eta\in(0,\pi/2]$, $a>0$, $\theta\in(0,\pi/2], r_e>0$.
Assume that $\S=\overline{\p\O\cap \mfR^{n+1}_+}$ meets $\p\mfR^{n+1}_+$ at the constant contact angle $\theta\in(0,\pi/2]$.
Let $O$ be the center chosen as in \eqref{defn-O-halfspace}. 
Then there exists some constant $C=C(n,\abs{T},\abs{\Gamma},d_\O,\eta, a, \theta,r_e)>0$ such that
\begin{align*}
    \rho_e-\rho_i
    \leq C
\begin{cases}
\Norm{H-H(\theta,\O)}^{\frac{1}{n+1}}_{1,\S},\quad&n\geq2,\\
\Norm{H-H(\theta,\O)}^{\frac{1}{2}}_{1,\S}\max\left\{\log\left(\Norm{H-H(\theta,\O)}^{-\frac{1}{2}}_{1,\S}\right),1\right\},
\quad&n=1,
\end{cases}
\end{align*}
where $H$ is the mean curvature function of $\S$ and $H(\theta,\O)$ is the $\theta$-reference mean curvature given by \eqref{defn-ref-meancurvature1}.
\end{theorem}

The proofs of \cref{Thm-stabi-OverDetermined-halfspace} and \cref{Thm-stabi-Meancurvature-halfspace} follow from Magnanini-Pogessi's idea \cite{MP22} by utilizing the integral identities \eqref{int-iden1} and \eqref{iden-integral-meancurvature-2}. We remark that in \cite{MP22}, the constant in the quantitative stability estimate depends on the Lipschitz norm for the solution of the corresponding mixed boundary value problem. In our proof, we derive the gradient estimate for the solution to \eqref{eq-mixed-halfspace0} in terms of geometric quantities of $\O$,  by utilizing the uniform exterior $\theta$-spherical cap condition. As a result, the constants in \cref{Thm-stabi-OverDetermined-halfspace} and \cref{Thm-stabi-Meancurvature-halfspace} depend only on geometric quantities of $\O$.

\

\noindent{\bf Organization of the paper.}
In \cref{Sec-2}, we recall some useful tools for the study of capillary problems and introduce some capillary geometric quantities.
In \cref{Sec-3} we prove the key integral identities and consequently, we prove the rigidity result (\cref{Thm-rigid-BVP-halfspace}).
In \cref{Sec-4} we first establish some PDE requisites for our purpose and then we prove the quantitative stability results (\cref{Thm-stabi-OverDetermined-halfspace}, \cref{Thm-stabi-Meancurvature-halfspace}).
As a by-product, the stability result for the Heintze-Karcher-type inequality for capillary hypersurfaces (\cref{Thm-stabi-HK-halfspace}) is also achieved.


\section{Preliminaries}\label{Sec-2}
Let $\O\subset\mfR^{n+1}_+$ be a bounded domain with piece-wise smooth boundary $\p\O=\S\cup T$, where $\S:=\overline{\p\O\cap \mfR^{n+1}_+}$ is a compact smooth hypersurface in $\overline{ \mfR^{n+1}_+}$, $T:=\p\Omega\cap\p \mfR^{n+1}_+$ and $\S$ meets $T$ at an $(n-1)$-dimensional submanifold $\Gamma:=\S\cap T$. 
We use the following notation for normal vector fields. Let $\nu$ and $\bar N:=-E_{n+1}$ be the outward unit normal to $\S$ and $\p\mfR^{n+1}_+$ (with respect to $\O$) respectively.  Let $\mu$ be the outward unit co-normal to $\Gamma\subset \S$ and $\bar \nu$ be the outward unit co-normal to $\Gamma\subset T$. Under this convention, along $\Gamma$ the bases $\{\nu,\mu\}$ and $\{\bar \nu,\bar N\}$ have the same orientation in the normal bundle of $\p \S\subset \mfR^{n+1}$. In particular, we say that $\S$ meets $\p\mfR^{n+1}_+$ with a fixed contact angle $\theta$ if along $\Gamma$, \begin{align}\label{munu}
		&&\mu=\sin \theta \bar N+\cos\theta \bar \nu,\quad \nu=-\cos \theta \bar N+\sin \theta \bar \nu.
	\end{align}	
	We denote by $\bar \n$, $\bar \Delta$, $\bar \n^2$ and $\bar{\rm div}$, the gradient, the Laplacian, the Hessian and the divergence on $\mfR^{n+1}$ respectively, while by $\n$, $\Delta$, $\n^2$ and ${\rm div}$, the gradient, the Laplacian, the Hessian and the divergence on the smooth part of $\p \O$, respectively. Let $g$, $h$ and $H$ be the first, second fundamental forms and the mean curvature  of the smooth part of $\p \O$ respectively. Precisely, $h(X, Y)=\<\bar\n_X\nu, Y\>$ and $H={\rm tr}_g(h)$.

We recall the classical $P$-function related to the solution $f$ of \eqref{eq-mixed-halfspace0}. Precisely, letting $$P=\frac{1}{2}|\bar \nabla f|^2- \frac{1}{n+1}f,$$ a direct computation gives
  \begin{align*}
   \bar\Delta P
   =|{\bar\nabla}^2 f|^2-\frac{(\bar\Delta f)^2}{n+1}.
 \end{align*}
 Using the Cauchy-Schwarz inequality, one finds
 \begin{align*}
     \bar\De P\geq0,
 \end{align*}
 and equality holds if and only if  $\bar\nabla^2f$ is proportional to the identity matrix.

We need the following version of Reilly's formula, we refer to \cite{JXZ22} and also \cite{Reilly77} for the details of the proof and the derivation of these formulas.
\begin{proposition}[{\cite{Reilly77}, \cite[Proposition 2.1]{JXZ22}}]
Let $f\in W^{1,\infty}(\O)\cap W^{2,2}(\O)$. We have
\begin{align}\label{formu-reilly1}
			&&\int_\O  (\bar\Delta f)^2-\left|\bar\n^2 f\right|^2 \rd x
			=\int_{\p\O}  f_\nu \De f-g(\n f_\nu, \n f) +H f_\nu^2 +h(\n f, \n f)\rd A. 
\end{align}
\end{proposition}

We recall the following PDE lemmas that are essentially used in the quantitative analysis, which are derived in \cite{MP22}. The statement we adopt here is a slightly simplified version of \cite[Corollary 3.2, Lemma 3.4]{MP22} (by taking $\alpha=0,1; p=r=2; q=\infty$), which are sufficient for our purposes.We shall use $\|\cdot\|_{k,\O}$ to denote $L^k$-norm on $\O$ for $k=2,\infty$, and set $\delta_{\p\O}(x):={\rm dist}(x,\p\O)$.

{\color{black}

\begin{lemma}[{\cite[Corollary 3.2]{MP22}}, Hardy-Poincar\'e-type inequality]\label{Cor-MP22-3-2}
Let $\O\subset\mfR^{n+1}$ be a bounded domain satisfying the $(\eta,a)$-uniform interior cone condition for some $\eta\in(0,\pi/2], a>0$.
Let $\alpha=0$ or $1$. For a function $\phi$ satisfying $$\bar\nabla\phi \in L^1_{loc}(\O),\quad \delta_{\partial \O}^\alpha\bar\nabla^2\phi\in L^2(\O),$$
and
\begin{align*}
    \int_{\O}\bar\nabla \phi \rd x=\vec0,
\end{align*}
it holds that
\begin{align}\label{ineq-Hessian-phi}
    \left\|\bar\nabla \phi\right\|_{2,\O}
    \leq C(n, d_\O, \eta,a)\left\|\delta^\alpha_{\partial \O} \bar\nabla^2 \phi\right\|_{2,\O}.
\end{align}
\end{lemma}

\begin{lemma}[{\cite[Lemma 3.4]{MP22}}]\label{Lem-MP22-Lem-3-4}
Let $\O\subset\mfR^{n+1}$ be a bounded domain satisfying the $(\eta,a)$-uniform interior cone condition for some $\eta\in(0,\pi/2], a>0$.
For any $\phi\in W^{1,\infty}(\O)$, there holds
\begin{equation}\label{ineq-maxphi-minphi}
    \max_{\overline \O}\phi-\min_{\overline \O}\phi
    \leq C(n, d_\O,\eta,a)
\begin{cases}
\left\|\bar\nabla \phi\right\|_{2,\O}^{\frac{2}{n+1}}\left\|\bar\nabla \phi\right\|_{\infty,\O}^{\frac{n-1}{n+1}},&\quad n\geq2,\\
\norm{\bar\nabla\phi}_{2,\O}\log\left(e\abs{\O}^{\frac12}\frac{\norm{\bar\nabla\phi}_{\infty,\O}}{\norm{\bar\nabla \phi}_{2,\O}}\right),
&\quad n=1.
\end{cases}
\end{equation}
\end{lemma}
}

\section{Integral Identities and Rigidity for Partially Overdetermined Problem}\label{Sec-3}

Let us begin by recalling some basic facts that are proved by integration by parts. See for example \cite{AS16} or \cite[Proposition 4.1]{JXZ22}.

\begin{proposition}
    It holds 
\begin{align}\label{eq-conservation-halfspace}
	\int_{\S}\left<\nu, E_{n+1}\right>\rd A=|T|.
\end{align}
Assume that $\S$ meets $\p\mfR^{n+1}_+$ with a constant contact angle $\theta\in(0,\pi)$, then we have
\begin{align}\label{eq-balancing-halfspace}
\int_{\S}H \left<\nu, E_{n+1}\right>\rd A=\sin\theta|\Gamma|,    
\end{align}
and the Minkowski-type formula
\begin{align}\label{eq-Minkowski-halfspace}
    \int_\S [n(1-\cos\theta\<\nu, E_{n+1}\>) -H\<x,\nu\>] \rd A=0.
\end{align}
\end{proposition}

Our first observation is that, given  $\O$ of the form $\O=B_r(z)\cap\mfR^{n+1}_+$, the related quadratic function $f(x)=\frac{\vert x-z\vert^2-r^2}{2(n+1)}$ solves \eqref{eq-mixed-halfspace0} for some $c\in\mfR$ with
\begin{align}\label{eq-f-nu}
    f_\nu\equiv\frac{r}{n+1}\text{ on }\S,
\end{align}
and a simple integration by parts yields
\begin{align}\label{eq-R1}
  \frac{r}{n+1}|\S|
  =\int_{\S}f_{\nu}\rd A
  =\int_{\O}\bar\Delta f \rd x+\int_{T}f_{\bar N}\rd A=\vert\O\vert-c\vert T\vert.
\end{align}
Enlightened by this,
we set for a general $\O$ and some $c\in \mfR$,
\begin{align}\label{defn-ref-radius1}
R(c,\O)\coloneqq(n+1)\frac{\vert\Omega\vert-c\vert T\vert}{\vert\S\vert},
\end{align}
which will be referred to as \textit{the reference radius} of $\O$.

Concerning the capillary phenomenon in the half-space, we fix a capillary contact angle $\theta\in(0,\pi)$ and observe that:
for the model case in the upper half-space, i.e., $$\O=B_{r}(-r\cos\theta E_{n+1})\cap \mfR_{+}^{n+1},$$ it is easy to see that $f(x)=\frac{\vert x-(-r\cos\theta E_{n+1})\vert^2-r^2}{2(n+1)}$ solves  \eqref{eq-mixed-halfspace0} on $\O$ with a specific number $c=-\frac{r\cos\theta}{n+1}$.
Moreover, since $\S$ is spherical, we can rewrite the radius $r$ as:
\begin{align*}
r
=\frac{n}{H}
=\frac{n\vert T\vert}{\sin\theta\vert\Gamma\vert},
\end{align*}
where \eqref{eq-conservation-halfspace} and \eqref{eq-balancing-halfspace} are used to derive the last equality.
This motivates the following definition:
\begin{definition}
\normalfont
Given $\theta\in(0,\pi)$, 
we define a constant that depends on the choice of $\theta$ and the domain $\O$ by
\begin{align}\label{defn-c1}
c(\theta,\O)\coloneqq-\frac{n}{n+1}\cot\theta\frac{  |T|}{|\Gamma|}.
\end{align}
\end{definition}
\begin{remark}\label{Rem-c_R}
   \normalfont
Note that for $\O\subset\mfR^{n+1}_+$ such that $\S$ has a fixed contact angle $\theta\in(0,\pi)$ with $\p\mfR^{n+1}_+$, a specific mixed boundary equation \eqref{eq-mixed-halfspace0} with $c=c(\theta,\O)$ is very useful for the study of capillary phenomenon in the half-space. Indeed, an integral approach to the Alexandrov-type theorem is built on such mixed boundary equation, see the proof of \cite[Theorem 1.1]{JXZ22}.

Enlightened by this and in view of \eqref{eq-f-nu}, \eqref{eq-R1},
we set
\begin{align}\label{defn-ref-radius11}
    R(\theta,\O)\coloneqq R(c(\theta,\O),\O)
    =(n+1)\frac{\vert\Omega\vert-c(\theta,\Omega)\vert T\vert}{\vert\S\vert}.
\end{align}
We call such $R(\theta,\O)$ the \textit{$\theta$-reference radius}, and naturally the \textit{$\theta$-reference mean curvature} is defined as
\begin{align}\label{defn-ref-meancurvature1}
    H(\theta,\O)\coloneqq\frac{n}{R(\theta,\O)}.
\end{align}
\end{remark}

We end this subsection by showing that, as claimed in the introduction, if $\S$ has constant mean curvature and meets $\p\mfR^{n+1}_+$ with a fixed contact angle $\theta$, then the mean curvature must be $H(\theta,\Omega)$.
\
\begin{proposition}\label{Prop-H1}
        Given $\theta\in(0,\pi)$,
        let $\O\subset\mfR^{n+1}_+$ such that $\S$ is a CMC hypersurface having fixed contact angle $\theta$ with $\p\mfR^{n+1}_+$.
        There holds that
        \begin{align}\label{eq-H-H1}
            H
            =\frac{n}{n+1}\frac{\vert\S\vert}{\vert\Omega\vert-c(\theta,\O)\vert T\vert}
            =H(\theta,\O).
        \end{align}
\end{proposition}
\begin{proof}
Considering the position vector field $X(x)=x$, a simple integration by parts gives
\begin{align}\label{eq-div-X-1}
    (n+1)\vert\Omega\vert
    =\int_\Omega\bar{\rm div}X(x)\rd x
    =\int_\S\left<x,\nu\right>\rd A.
\end{align}
Since $H$ is a constant, we see from the Minkowski-type formula \eqref{eq-Minkowski-halfspace} that
\begin{align}\label{eq-|S|-1}
    \vert\S\vert
    =&\frac{H\int_\S\left<x,\nu\right>\rd A+n\cos\theta\int_\S\left<\nu,E_{n+1}\right>\rd A}{n}\notag\\
    =&\frac{(n+1)H\vert\Omega\vert+n\cos\theta\vert T\vert}{n},
\end{align}
where we have used \eqref{eq-conservation-halfspace} and \eqref{eq-div-X-1} for the last equality.

On the other hand, thanks again to the fact that $H$ is a constant, \eqref{eq-conservation-halfspace} and \eqref{eq-balancing-halfspace} imply that
\begin{align}\label{eq-H-T-Gamma}
    H=\frac{\sin\theta\vert\Gamma\vert}{\vert T\vert},
\end{align}
thus we may rewrite the second term in \eqref{eq-|S|-1} to be
\begin{align*}
    n\cos\theta\vert T\vert
    =n\cot\theta\frac{\vert T\vert^2}{\vert\Gamma\vert}H,
\end{align*}
which in turn gives that
\begin{align*}
    \vert \S\vert=\frac{n+1}{n}H\left(\vert\Omega\vert+\frac{n}{n+1}\cot\theta\frac{\vert T\vert^2}{\vert\Gamma\vert}\right).
\end{align*}
Recalling \eqref{defn-c1}, we thus find $H=H(\theta,\O)$.
\end{proof}
For $c\in\mfR$, let 
 $$g(x)=\frac{|x-\tilde{O}|^2}{2(n+1)},$$
 where {\color{black} $\tilde{O}=(0,\ldots,0,(n+1)c)$}.
 It is easy to check that
 \begin{align}\label{eq-derivatives-g}
    \bar\nabla g(x)
    =\frac{x-\tilde{O}}{n+1},\quad g_{ij}=\frac{\delta_{ij}}{n+1},
\end{align}
and from a simple computation we observe that $g$ satisfies
 \begin{equation}
\begin{cases}
  \bar\Delta g=1 &\text{in } \Omega,\\
  \bar \n_{\bar N}g=c
  &\text{on }T.
\end{cases}
 \end{equation}
We have the following integral identity.
\begin{proposition}
Let {\color{black}$f\in W^{1,\infty}(\O)\cap W^{2,2}(\O)$} be a solution to  \eqref{eq-mixed-halfspace0}. Then we have: for any $R\in\mfR$, there holds
 \begin{align}\label{eq-integral-identity-1}
 \int_{\Omega}(-f)&\left\{\vert\bar\nabla^2f\vert^2-\frac{(\bar\De f)^2}{n+1}\right\}\rd x\notag\\
 =&\frac12\int_\Sigma\left(|\bar\nabla f|^2-\left(\frac{R}{n+1}\right)^2\right)(f_{\nu}-g_{\nu})\rd A\notag\\
 =&\frac12\int_\Sigma\left(f_\nu^2-\left(\frac{R}{n+1}\right)^2\right)\bigg(f_{\nu}-\Big\<\frac{x+(n+1)cE_{n+1}}{n+1},\nu\Big\>\bigg)\rd A.
 \end{align}
\end{proposition}

\begin{proof}
Integrating by parts, we get
 \begin{align}\label{eq1}
 \int_\O(-f)\bar\Delta P\rd x
 =-\int_{\p\O} f P_{\nu}\rd A +\int_{\p\O} Pf_{\nu}\rd A -\int_\O P\rd A.
 \end{align}
In terms of the second term on the RHS of \eqref{eq1}, we use integration by parts again to find
\begin{align}\label{eq-second term1}
  \int_{\p\O} P f_{\nu}\rd A
  &=\int_{\p\O} P (f_\nu-g_\nu)\rd A+\int_{\p\O} P g_\nu\rd A\notag\\
   &=\int_{\p\O} P (f_\nu-g_\nu)\rd A+\int_{\O} P\rd x+\int_\O \<\bar\nabla P,\bar\nabla g\>\rd A\notag\\
  &=\int_{\p\O} P (f_\nu-g_\nu)\rd A+\int_\O P\rd x + \int_\O f_j f_{ji}g_i\rd x-\frac1{n+1}\int_\O \<\bar\nabla f,\bar\nabla g\>\rd x,
\end{align}
where by virtue of \eqref{eq-mixed-halfspace0} and \eqref{eq-derivatives-g},
\begin{align*}
\int_\O f_j f_{ji}g_i\rd x
&=\int_{\p\O} f g_i f_{ji}\nu_{j}\rd A -\int_{\O} f f_{jij}g_i\rd x-\int_\O f f_{ji} g_{ij}\rd x \notag\\
&=\int_{\p\O} f g_j f_{ji}\nu_i\rd A-\frac1{n+1}\int_\O f\rd x,
\end{align*}
and notice that
\begin{align*}
\int_\O \<\bar\nabla f,\bar\nabla g\>\rd x
= \int_{\p\O}f g_\nu \rd A-\int_{\O} f\rd x,
\end{align*}
\eqref{eq-second term1} thus reads as
\begin{align}\label{eq-second term2}
  \int_{\p\O} P f_{\nu}\rd A
  &=\int_{\p\O} P (f_\nu-g_\nu)\rd A+\int_\O P\rd x + \int_{\p\O} f f_{ji}g_j\nu_{i}\rd A-\frac1{n+1}\int_{\p\O} f g_{\nu}\rd A.
\end{align}
Substituting \eqref{eq-second term2} into \eqref{eq1}, we obtain
\begin{align}\label{eq2}
\int_{\Omega}&(-f)\bar\De P \rd x\notag\\
&= -\int_{\p\O} f P_{\nu}\rd A+\int_{\p\O}P(f_\nu-g_\nu)\rd A+ \int_{\p\O} f f_{ji}g_j\nu_{i}\rd A-\frac1{n+1}\int_{\p\O} f g_{\nu}\rd A\notag\\
&=-\int_{\p\O}f(f_j-g_j) f_{ji}\nu_i\rd A+\frac12\int_{\p\O}|\bar\nabla f|^2 (f_\nu-g_\nu)\rd A\notag\\
&=-\int_{\p\O}f(f_j-g_j) f_{ji}\nu_i\rd A+\frac12\int_{\p\O}\left(|\bar\nabla f|^2-\left(\frac{R}{n+1}\right)^2\right) (f_\nu-g_\nu)\rd A,
\end{align}
where we expand $P$ and $P_\nu$ to derive the second equality; the second and the last equality hold thanks to the fact that
\begin{align*}
\int_{\p\O}(f_\nu-g_{\nu})\rd A
=\int_\O (\bar\Delta f-\bar\Delta g)\rd x
=0.
\end{align*}
Recalling that $f=0$ on $\S$, $f_{n+1}=g_{n+1}=-c
$ on $T$, we have
\begin{align}\label{eq-boundary1}
-\int_{\p\O}f(f_j-g_j) f_{ji}\nu_i\rd A
&=\int_T f(f_j-g_j)f_{n+1,j}\rd A\notag\\
&=\int_Tf f_{n+1,n+1}(f_{n+1}-g_{n+1})\rd A=0.
\end{align}
Substituting \eqref{eq-boundary1}  back into \eqref{eq2},
we obtain the second equality of \eqref{eq-integral-identity-1};
the last equality of \eqref{eq-integral-identity-1} follows simply by expanding $g_\nu$ on $\S$.
This completes the proof.
\end{proof}

\begin{proof}[Proof of \cref{Thm-rigid-BVP-halfspace}]
Since $f_\nu\equiv c_0$ on $\S$, using integration by parts we observe that
\begin{align*}
  c_0\vert\S\vert
  =\int_{\S}f_{\nu}\rd A
  =\int_{\O}\bar\Delta f \rd x+\int_{T}f_{n+1}dA=|\O|-c|T|,
\end{align*}
it follows that the constant is in fact characterized by
$c_0=f_\nu=\frac{\vert\O\vert-c\vert T\vert}{\vert\S\vert}$ on $\S$.

Taking the integral identity \eqref{eq-integral-identity-1} into consideration with $R=R(c,\O)=(n+1)c_0$, the Cauchy-Schwarz inequality then implies that $\bar\nabla^2f$ is proportional to the identity matrix.
Since $\bar\De f=1$, we know that $f$ is a quadratic function of the form $f(x)=\frac{\vert x-z\vert^2}{2(n+1)}+C$ for some $z=(z',z_{n+1})\in\mfR^{n}\times\mfR$ and $C\in\mfR$.

By the connectedness of $\O$ and the fact that $f=0$ on $\S$, we conclude that $\S$ is spherical and hence $\O$ is the part of a ball centered at $z$ with radius $r$ given by
\begin{align*}
    r=\vert x-z\vert=\sqrt{-2(n+1)C}\text{ on }\S.
\end{align*}
Moreover, since $c_0=f_\nu$ on $\S$, we find: at any $x\in\S$,
\begin{align*}
    c_0=f_\nu(x)
    =\frac{\vert x-z\vert}{n+1},
\end{align*}
which yields that $r=(n+1)c_0
=(n+1)\frac{\vert\O\vert-c\vert T\vert}{\vert\S\vert}$.

On the other hand, since $\S$ is spherical, the contact angle of $\S$ with the plane $\p\mfR^{n+1}_+$ must be a constant, say $\theta$.
It follows from \eqref{eq-H-H1} that
\begin{align*}
    \frac{n}{r}=
    H=H(\theta,\Omega)
    =\frac{n}{n+1}\frac{\vert\S\vert}{\vert\O\vert-c(\theta,\O)\vert T\vert},
\end{align*}
a direct computation then yields that
\begin{align*}
    c=c(\theta,\O)=-\frac{n}{n+1}\cot\theta\frac{\vert T\vert}{\vert\Gamma\vert},
\end{align*}
using \eqref{eq-balancing-halfspace} and \eqref{eq-conservation-halfspace}, we may rearrange this to see that $\theta$ can be characterized by
\begin{align*}
    \cos\theta=-c\cdot\frac{n+1}{r}=-\frac{c}{c_0}.
\end{align*}

Finally, using the fact that $-f_{n+1}=c$ on $T\subset\p\mfR^{n+1}_+$, we get: $z_{n+1}=(n+1)c$.
Therefore, we conclude that $\O$ and $f$ have all the desired properties, which completes the proof.
\end{proof}

 \

The second integral identity concerns with the mean curvature of $\S$. It can be viewed as a reformulation of Reilly's formula \eqref{formu-reilly1}.
\begin{proposition}
Let $\Omega\subset\mfR^{n+1}_+$ be such that $\S$ intersects $\p\mfR^{n+1}_+$ at a constant contact angle $\theta\in(0,\pi)$.
Let $f\in W^{2,2}(\O)\cap W^{2,1}(T)\cap C^2(\bar \Omega\setminus\Gamma) \cap C^1(\bar \Omega)$  be a solution to  \eqref{eq-mixed-halfspace0} with  $c=c(\theta,\O)$,
then
we have
\begin{align}\label{iden-integral-meancurvature-1}
\int_\Omega\left\{\vert\bar\nabla^2f\vert^2-\frac{(\bar\De f)^2}{n+1}\right\}\rd x+\frac{n}{n+1}\frac{1}{\tilde R}\int_\S(f_\nu-\tilde R)^2\rd A
=\int_\S(H(\theta,\O)-H)f_\nu^2\rd A,
\end{align}
  where $\tilde R\coloneqq\frac{R(\theta,\O)}{n+1}$,
  $R(\theta,\O)$ and $H(\theta,\O)$ are the $\theta$-reference radius and $\theta$-reference mean curvature given by \eqref{defn-ref-radius11} and \eqref{defn-ref-meancurvature1} respectively.
\end{proposition}
\begin{proof}

We begin with the notification that {the regularity assumption $f\in C^{1}{(\overline{\Omega})}\cap C^2(\bar \Omega\setminus\Gamma)\cap W^{2,1}(T)$ is sufficient for the integration by parts
\begin{align*}
\int_T \Delta f \rd A =\int_\Gamma \<\nabla^T f,\bar\nu\>
\end{align*}
to hold.
In fact, this is probably known to experts, see for example \cite[Lemma 4.3]{CL22}.

}

Since $f$ solves \eqref{eq-mixed-halfspace0}, 
  we can argue as in \cite[(27) (28)]{JXZ22} to find
  \begin{align}\label{eq-JXZ22-28}
      \int_T\De f\rd A
      =-c\tan\theta\vert\Gamma\vert.
  \end{align}
  Exploiting the Reilly's formula \eqref{formu-reilly1} and using \eqref{eq-JXZ22-28}, we obtain
  \begin{align}\label{eq-LaplaceP}
      \int_\Omega\bar\De P\rd x
      =&\int_\Omega\vert\bar\n^2f\vert^2-\frac{1}{n+1}(\bar\De f)^2\rd x\notag\\
      =&\frac{n}{n+1}\vert\Omega\vert+c^2\tan\theta\vert\Gamma\vert-\int_\S Hf_\nu^2\rd A.
  \end{align}
  Concerning the last term, we first notice that a simple integration by parts yields
  \begin{align}
      \vert\Omega\vert=\int_\Omega\bar\De f\rd x
      =\int_\S f_\nu \rd A+c\vert T\vert,
  \end{align}
  and hence
  \begin{align}
      \frac{1}{\tilde R}\int_\S f_\nu^2\rd A
      =&\frac{1}{\tilde R}\int_\S(f_\nu-\tilde R)^2\rd A+2\int_\S f_\nu \rd A-\tilde R\vert\S\vert\notag\\
      =&\frac{1}{\tilde R}\int_\S(f_\nu-\tilde R)^2\rd A+2\vert\Omega\vert-2c\vert T\vert-\tilde R\vert\S\vert\notag\\
      =&\frac{1}{\tilde R}\int_\S(f_\nu-\tilde R)^2\rd A+\vert\Omega\vert-c\vert T\vert,
  \end{align}
  where we used the definition of $\tilde R$ for the last equality.

  Since $H(\theta,\O)=\frac{n}{n+1}\frac{1}{\tilde R}$ is a constant, it follows that
  \begin{align}
      &\int_\S Hf_\nu^2\rd A
      =H(\theta,\O)\int_\S f_\nu^2\rd A+\int_\S(H-H(\theta,\O))f_\nu^2\rd A\notag\\
      =&\frac{n}{n+1}\left(\frac{1}{\tilde R}\int_\S(f_\nu-\tilde R)^2\rd A+\vert\Omega\vert-c\vert T\vert\right)+\int_\S(H-H(\theta,\O))f_\nu^2\rd A\notag.
  \end{align}
  Putting this back into \eqref{eq-LaplaceP} and recalling \eqref{defn-c1}, we get \eqref{iden-integral-meancurvature-1}.
\end{proof}
\begin{remark}
    \normalfont
    In this integral identity, we do not restrict the range of $\theta$, we shall point out that for the case $\theta\in(0,\pi/2]$, one may use the regularity theorem of mixed boundary elliptic equations (see e.g. \cite{Liebermann86,Liebermann89})
    to obtain a regular enough solution, this fact has been essentially used in \cite[Proof of Theorem 1.1]{JXZ22}.
\end{remark}



\section{Quantitative Stability Results}\label{Sec-4}


\subsection{Strategy for proving the quantitative stability results}\label{Sec-4-1}
Our strategy for proving the stability is to show that $\S$ is nearly spherical, which is done by showing that $\rho_e-\rho_i$ can be controlled by an arbitrarily small quantity.

Precisely, we compare the solution $f$ to the mixed boundary value problem \eqref{eq-mixed-halfspace0} with the quadratic function $h=\frac{|x-O|^2}{2(n+1)}$.
Denote by $\phi=h-f$ the difference of the functions, it is easy to see that
\begin{align}
\max_{x\in\S}\phi(x)-\min_{x\in\S}\phi(x)=\frac1{2(n+1)}(\rho_e+\rho_i)(\rho_e-\rho_i).
\end{align}
Moreover, from the definition of $\rho_e$, we can use the triangle inequality to find that $d_{\O}\leq 4\rho_e$, see \cite[Lemma 4.3]{MP22}.
Thus we get
\begin{align}\label{eq-rhoe-rhoi}
  \rho_e-\rho_i\leq &\frac{2(n+1)}{\rho_{e}}(\max_{x\in\S}\phi(x)-\min_{x\in\S}\phi(x))\notag\\
  \leq &\frac{8(n+1)}{d_{\Omega}}(\max_{x\in\bar\O}\phi(x)-\min_{x\in\bar\O}\phi(x)).
\end{align}

This is the starting point for the quantitative stability result. If we could control the oscillation of the function $\phi$, then we may measure how close shall $\S$ be to being spherical.
To this end, we shall exploit \cref{Lem-MP22-Lem-3-4} and \cref{Cor-MP22-3-2}.
Note that on the right of \eqref{ineq-maxphi-minphi}, the $L^\infty$-norm of $\bar\nabla\phi$ can be easily estimated by the Lipschitz norm of $f$ and the geometric quantities of $\O$, and we can control the $L^2$-norm of $\bar\nabla\phi$ by virtue of \eqref{ineq-Hessian-phi}.
Our choice of the center $O$ is thus clear, we want $\bar\nabla\phi$ to have a vanishing average on $\O$, to this end, we choose $O$ to be
\begin{align*}
  \frac1{|\O|}\bigg(\int_{\O}x\rd x-(n+1)\int_{T}f\rd A\cdot E_{n+1}\bigg).
\end{align*}
A simple computation gives
\begin{align*}
  \int_\O\bar\nabla\phi\rd x
  =\int_{\O}(\bar\nabla h-\bar\nabla f)\rd x
  =&\int_{\O}(\frac{x-O}{n+1}-\bar\nabla f) \rd x\notag\\
  =&\int_{T}f\rd A\cdot E_{n+1}-\int_{\O}\bar\nabla f\rd x
  =0.
\end{align*}
A direct consequence of the choice of $\O$ is the following point-wise estimate of $\bar\nabla\phi$: for any $y\in\overline\O$,
\begin{align}\label{eq-Lipschitz-phi}
     |\bar\nabla \phi|(y)
     &=\left|\frac{1}{n+1}(y-O)-\bar\nabla f(y)\right|
     \leq\frac{1}{n+1}|y-O|+|\bar\nabla f(y)|\notag\\
     &\leq\frac{1}{n+1}\cdot\frac{1}{|\Omega|}\int_{\Omega}|y-x|\rd x+
     \frac{1}{|\Omega|}\int_{\Omega}|\bar\nabla f(x)|\rd x+L
     \leq\frac{d_{\Omega}}{n+1}+2L.
 \end{align}
On the other hand, notice that $\bar\nabla^2h=\frac{1}{n+1}{\rm Id}$, $\bar\De f=1$ on $\O$, and hence
\begin{align*}
    \vert\bar\nabla^2\phi\vert^2
    =\vert\bar\nabla^2f-\frac{1}{n+1}{\rm Id}\vert^2
    =\vert\bar\nabla^2f-\frac{\bar\De f}{n+1}{\rm Id}\vert^2
    =\vert\bar\nabla^2f\vert^2-\frac{(\bar\De f)^2}{n+1}.
\end{align*}
Combining these facts, we deduce a quantitative control on the proximity of $\S$ to a sphere as follows:

For $n\geq2$,
\begin{align}\label{ineq-rhoe-rhoi}
    \rho_e-\rho_i
    \leq C(n,d_\O,\alpha,\eta,a)\frac{8(n+1)}{d_\O}\left(\int_\O\delta_{\p\O}^{2\alpha}\bar\De P\rd x\right)^{\frac{1}{n+1}}\left(\frac{d_\O}{n+1}+2L\right)^{\frac{n-1}{n+1}},
\end{align}
and for $n=1$,
\begin{align}\label{ineq-rhoe-rhoi-n=1}
\rho_e-\rho_i
\leq& C(d_\O,\eta,a)\norm{\bar\nabla\phi}_{2,\O}\max\left\{\log\left(\frac{\norm{\bar\nabla\phi}_{\infty,\O}}{\norm{\bar\nabla\phi}_{2,\O}}\right),1\right\}\notag\\
\overset{\eqref{ineq-Hessian-phi}}{\leq}&C(d_\O,\eta,a)\norm{\delta_{\p\O}^\alpha(\bar\De P)^\frac12}_{2,\O}\max\left\{\log{\frac{(n+1)^{-1}d_\O+2L}{\norm{\delta_{\p\O}^\alpha\bar\De P}_{2,\O}}},1\right\},
\end{align}
where the dependence on $\abs{\O}$ is replaced by $d_\O$ thanks to the isodiametric inequality.



\subsection{Gradient estimate of the mixed boundary equations}\label{Sec-4-2}
In this subsection, we aim at providing a geometric control of the Lipschitz bound $L$ in \eqref{ineq-rhoe-rhoi}.

\begin{proposition}[$C^0$-estimate of \eqref{eq-mixed-halfspace0}]\label{Prop-C^0-halfspace}
Let $\O\subset\mfR^{n+1}_+$ and $f\in W^{1,\infty}(\Omega)\cap W^{2,2}(\Omega)$ be a solution to \eqref{eq-mixed-halfspace0} with some $c\leq0$, then
\begin{align}\label{esti-f-C0}
    \left\|f\right\|_{\infty,\O}
    \leq d_\Omega^2+(n+1)^2c^2.
\end{align}
\end{proposition}
\begin{proof}
    We begin by introducing the following auxiliary function $g(x)=\frac{\vert x-O\vert^2-R^2}{2(n+1)}$, where $O=(O',(n+1)c)$ for some $(O',0)\in T$ and 
{$R=d_\Omega+(n+1)c$.}
    It is easy to see that $g$ satisfies
    \begin{align*}
    \begin{cases}
        \bar\De g=1\quad&\text{in }\Omega,\\
        \bar\nabla_{\bar N}g=c\quad&\text{on }T.
    \end{cases}
    \end{align*}
    Moreover, by virtue of the triangle inequality, we find: for any $x\in\Omega$, $${\rm dist}(x,O)
    \leq{\rm dist}(x,x_0)+{\rm dist}(x_0,O)
    \leq d_\Omega+(n+1)c,$$
    and hence $g\leq0$ on $\S$ thanks to the definition of $R^2$.
Using the maximum principle (see e.g. \cite[Lemma 4.1]{Pog22}) for $g-f$ and $f$ respectively, we see that for any $x\in\Omega$,
    \begin{align*}
        -\frac{R^2}{2(n+1)}\leq\frac{\vert x-O\vert^2-R^2}{2(n+1)}=g(x)\leq f(x)\leq0.
    \end{align*}
    The assertion follows easily.
\end{proof}

Next we shall utilize the $r_0$-uniform exterior $\theta$-spherical cap condition in \cref{Defn-exterior-sphere} to prove the gradient estimate. The following proposition reveals the naturality of \cref{Defn-exterior-sphere}.
\begin{proposition}\label{nat-exterior-sphere}
    Let $\O\subset\mfR^{n+1}_+$  and  $\S=\overline{\p\Omega\cap\mfR^{n+1}_+}$.
   Assume that for every $x\in\Gamma$, the contact angle satisfies
    $0<\tilde\theta(x)\leq \theta$ 
 for some $\theta\in(0,\pi)$. then there exists some $r_0>0$ such that
    $\Omega$ satisfies the $r_0$-uniform exterior $\theta_0$-spherical cap condition.
    If in addition, $\S$ is strictly convex, then $r_0$ can be chosen to be any large constant.
\end{proposition}
\begin{proof}
Since $\S$ is a $C^2$-hypersurface, there exists a touching ball of radius $r(x)$ and center $O_x=x+r(x)\nu(x)$ for every $x\in\S$ { such that $\overline{ B_{r}(x+r\nu(x))}\cap \overline{\Omega}=\{x\}$}.
For any interior point $x\in \mathring{\Sigma}$, $x_{n+1}>0$, and clearly we can choose the radius $r(x)$ small enough to ensure that  Condition {\bf (2)} holds.
For any $x\in\p\S$, by the contact angle condition $\tilde\theta(x)\leq \theta$, we have
\begin{align*}
\<O_x,E_{n+1}\>=\<x+r(x)\nu(x),E_{n+1}\>=r(x)\<\nu(x),E_{n+1}\>=r(x)\cos\tilde\theta(x)\geq r(x)\cos\theta,
\end{align*}
which shows that on the boundary $\p\S$, Condition {\bf(2)} always holds thanks to our choice of center of the touching ball.

Set
\begin{equation*}
\rho(x)=\inf_{r>0}\bigg\{\frac1r\Big|
\begin{aligned}
 \exists  B_{r}(x+r\nu(x))&~\mathrm{  s.t. } ~ \overline{ B_{r}(x+r\nu(x))}\cap \overline{\Omega}=\{x\},\\
      &\left<x+r\nu(x),E_{n+1}\right>\geq r\cos\theta
      \end{aligned}
      \bigg\},
\end{equation*}
since $\S$ is of class $C^2$, we know that $\rho$ is well-defined and continuous on $\S$
 with $\rho(x)<+\infty$ for any $x\in\S$.
Due to the compactness of $\S$, we can find a uniform constant $\rho_0$ such that
$\sup_{x\in\S}\rho(x)=\rho_0<\infty$.
If $\rho_0>0$, we set $r_0=\frac1{\rho_0}$, and clearly $\Omega$ satisfies the $r_0$-uniform exterior $\theta$-spherical cap condition.
If $\rho_0=0$, then $r_0$ can be chosen as large as possible so that $\Omega$ satisfies the $r_0$-uniform exterior $\theta$-spherical cap condition.

{ Moreover, if $\S$ is strictly convex, the radius of the touching ball from outside can be arbitrarily large. On the other hand, it is well known that the Gauss map $\nu:\S\to \mfS^n$ of strictly convex hypersurface $\S$ is a homeomorphism.
Since $\<\nu(x),E_{n+1}\>\geq \cos\theta$ for any $x\in\p\S$, and there exists a point $x\in\S$ such that $\nu(x)=E_{n+1}$, we have
\begin{align*}
\nu(\Sigma)\subset \left\{\omega\in S^n:\<\omega,E_{n+1}\>\geq \cos\theta\right\},
\end{align*}
which implies
$\<x+r\nu(x),E_{n+1}\> \geq r\cos\theta$ for any $x\in\S$ and any $r>0$. This is to say that Condition {\bf(2)} always holds for any touching ball with center $O_x=x+r\nu(x)$. Thus $\O$ satisfies the $r_0$-uniform exterior $\theta$-spherical cap condition with arbitrarily large $r_0$.
 }

\end{proof}

\begin{proposition}[Gradient estimate of \eqref{eq-mixed-halfspace0}]\label{Prop-GE-Halfspace}
Let $\O\subset\mfR^{n+1}_+$ satisfies the $r_e$-uniform exterior $\theta$-spherical cap condition for some $r_e>0, \theta\in (0,\frac{\pi}{2}]$. Let $f\in W^{1,\infty}(\Omega)\cap W^{2,2}(\Omega)$  be a solution to \eqref{eq-mixed-halfspace0} with either
\begin{itemize}
\item[(1)] $c<0$ when $\theta\in(0,\pi/2)$, or 
\item[(2)] $c=0$ when $\theta=\pi/2$.
\end{itemize}
Then
\begin{align}
    \|\bar\nabla f\|_{\infty,\Omega}\leq C(n,c,r_e,\theta,d_\Omega).
\end{align}

\end{proposition}
\begin{proof}
The case for $\theta=\frac{\pi}{2}$ has been proved in \cite{Pog22}, similar to the closed case in \cite{MP19}.  
Here we prove it for any  $\theta\in(0,\frac{\pi}{2}]$. 

\noindent{\bf Step 1.}  Gradient estimate on $\S$.

For any $x\in\S$, there exists $O_x\in\mfR^{n+1}$ such that $ B_{r_e}(O_x)$ is the touching ball in \cref{Defn-exterior-sphere}.

Let $R>r_e$ to be chosen later.
Denote by $A_{r_e,R}\coloneqq B_R(O_x)\setminus B_{r_e}(O_x)$ the annulus domain and set $\kappa\coloneqq\frac{r_e}{R}$.

 As proved in \cite[Lemma 3.9]{MP19}, the function
    \begin{align*}
        w(y)\coloneqq
\begin{cases}
\frac14\left\{\vert y-O_x\vert^2-r_e^2+R^2(1-\kappa^2)\frac{\log\vert y-O_x\vert-\log r_e}{\log\kappa} \right\}& \text{for } n=1,\\
\frac1{2(n+1)}\left\{\vert y-O_x\vert^2-r_e^2+\frac{R^2(1-\kappa^2)}{1-\kappa^{n-1}}\left[(\frac{\vert y-O_x\vert}{r_e})^{1-n}-1\right]\right\}& \text{for } n\geq 2,
\end{cases}
    \end{align*}
  satisfies
    \begin{align*}
        \begin{cases}
            \bar\De w=1\quad&\text{in }A_{r_e,R},\\
            w=0\quad&\text{on }\p A_{r_e,R},\\
            w\leq0\quad&\text{in } A_{r_e,R}.
        \end{cases}
    \end{align*}
    A direct computation shows that
    \begin{align*}
        \bar\nabla w(y)=
\begin{cases}
\frac12 \frac{y-O_x}{\vert y-O_x\vert^{2}}\left(\vert y-O_x\vert^{2}+\frac{R^2(1-\kappa^2)}{2\log\kappa}\right) &\text{for }n=1,\\
        \frac{1}{n+1}\frac{y-O_x}{\vert y-O_x\vert^{n+1}}\left(\vert y-O_x\vert^{n+1}-\frac{n-1}{2}\frac{R^2r_e^{n-1}(1-\kappa^2)}{1-\kappa^{n-1}}\right)& \text{for } n\geq 2,
\end{cases}
    \end{align*}
    for simplicity denote by $M$ a positive constant such that
\begin{align*}
M^{n+1}\coloneqq
\begin{cases}
-\frac{R^2(1-\kappa^2)}{2\log\kappa}&\text{for }n=1,\\
\frac{n-1}{2}\frac{R^2 r_e^{n-1}(1-\kappa^2)}{1-\kappa^{n-1}}&\text{for } n\geq 2,
\end{cases}
\end{align*}
   thus  on $T$, there holds
    \begin{align}\label{eq-normal-deri-w}
        \bar\nabla_{\bar N}w
        =\left<\bar\nabla w,-E_{n+1}\right>
        =\frac{1}{n+1}\left<O_x,E_{n+1}\right>\left(1-\frac{M^{n+1}}{\vert y-O_x\vert^{n+1}}\right), \forall n\geq 1.
    \end{align}
 Set
    \begin{align*}
        t\coloneqq
        \begin{cases}
            1-\frac{(n+1)c}{r_e\cos\theta},\quad&\theta\in(0,\pi/2),\\
        1,\quad&\theta=\pi/2,
        \end{cases}
    \end{align*}
since $c<0$ for $\theta\in(0,\pi/2)$, we have $t\geq1$.
We wish to choose $R$ large enough so that
\begin{align}\label{ineq-value-R}
R>M>t^{\frac1{n+1}}(r_e+d_\Omega).
\end{align}
In fact, a possible candidate is
\begin{align*}
R=\max\{2r_e,\frac1{r_e},8t(r_e+d_\Omega)^2,2t^{\frac12}r_e^{\frac{1-n}{2}}(r_e+d_{\Omega})^{\frac{n+1}{2}} \}.
\end{align*}
Let us verify \eqref{ineq-value-R}:
since $R\geq 2r_e$, we have $0<\kappa\leq\frac12$. It is easy to check that
\begin{align*}
-\frac{1}{4\log\kappa}<-\frac{(1-\kappa^2)}{2\log\kappa}<1, \text{ and } \frac12<\frac{1-\kappa^2}{1-\kappa^{n-1}}< 2\text{ for }n\geq2.
\end{align*}

For the case $n=1$,
\begin{align*}
M
=\left(-\frac{R^2(1-\kappa^2)}{2\log\kappa}\right)^{\frac12}<R,
\end{align*}
and
\begin{align*}
M
>\frac12\left(-\frac{R^2}{\log\kappa}\right)^{\frac12}=\frac12\left(\frac{R^2}{\log R+\log \frac1{r_e}}\right)^{\frac12}\geq \frac12\left(\frac{R^2}{2\log R}\right)^{\frac12}>\frac12\left(\frac R2\right)^{\frac12}\geq t^{\frac1{2}}(r_e+d_\Omega),
\end{align*}
where the second and  last inequalities hold by the definition of R.

For the case $n\geq 2$,
\begin{align*}
M=
\left(\frac{n-1}{2}\frac{R^2 r_e^{n-1}(1-\kappa^2)}{1-\kappa^{n-1}}\right)^{\frac1{n+1}}
<\left((n-1)R^2{r_e}^{n-1}\right)^{\frac1{n+1}}
\leq\left(\frac{n-1}{2^{n-1}}\right)^{\frac1{n+1}}R<R,
\end{align*}
where the second inequality holds due to $r_e\leq \frac12R$, and
\begin{align*}
M>
\left(\frac{n-1}{4}r_e^{n-1}R^2 \right)^{\frac1{n+1}}\geq \left(\frac{1}{4}r_e^{n-1}R^2 \right)^{\frac1{n+1}}\geq t^{\frac1{n+1}}(r_e+d_\Omega),
\end{align*}
where the last inequality holds by the definition of R.

 Note that for any $y\in\Omega$, the triangle inequality shows that
    \begin{align*}
        {\rm dist}(y,O_x)
        \leq{\rm dist}(y,x)+{\rm dist}(x,O_x)
        \leq r_e+d_\Omega,
    \end{align*}
by \eqref{ineq-value-R} and the fact that $t\geq1$, we have $\Omega\subsetneqq B_{M}(O_x)\subsetneqq B_R(O_x)$, so that
    for any $y\in T$, we have $\vert y-O_x\vert\leq M{t^{-\frac1{n+1}}}$.
    On the other hand,
    since $\theta\in(0,\pi/2]$,
    we see from \cref{Defn-exterior-sphere} (2) that $\left<O_x,E_{n+1}\right>\geq r_e\cos\theta\geq0$.
    Combining these facts, we may use \eqref{eq-normal-deri-w} to find: on $T$,
    \begin{align*}
        \bar\nabla_{\bar N}w
        \leq\frac{1}{n+1}\left<O_x,E_{n+1}\right>\left(1-t\right)
        \leq\frac{1}{n+1}r_e\cos\theta(1-t)
        = c.
    \end{align*}
 Here we have used the definition of $t$ to derive the last inequality.
 In particular, we obtain that $w$ satisfies
    \begin{align*}
        \begin{cases}
            \bar\De w=1\quad&\text{in }\Omega,\\
            w\leq0\quad&\text{on }\S,\\
            \bar\nabla_{\bar N}w\leq c\quad&\text{on }T.
        \end{cases}
    \end{align*}
    Using  the maximum principle (see e.g. \cite[Lemma 4.1]{Pog22}) for $w-f$, we see that $w\leq f$ on $\overline\Omega$.
    Moreover, by construction $w-f$ attains its maximal at $x\in\S$, and hence for any $x\in\S\setminus\Gamma$, the Hopf's boundary point lemma shows that
    \begin{align*}
        0\leq \bar\nabla _\nu f(x)\leq \bar\nabla_\nu w(x)
    \end{align*}
  A direct computation yields that  
    \begin{align*}
        \nabla_\nu w(x)=
        \begin{cases}
-\frac12 \left(r_e+\frac{R^2(1-\kappa^2)}{2r_e\log\kappa}\right) &\text{for }n=1,\\
        -\frac{1}{n+1}\left(r_e-\frac{n-1}{2}\frac{1-\kappa^2}{1-\kappa^{n-1}}\frac{R^2}{r_e}\right)& \text{for } n\geq 2,
\end{cases}
    \end{align*}
    Since $f=0$ on $\S$, we obtain the gradient estimate along $\S$:
    \begin{align}
        \|\bar\nabla f\|_{\infty,\S}\leq C(n,c,r_e,\theta,d_\Omega).
    \end{align}

    \noindent{\bf Step 2. }Gradient estimate on $\overline\Omega$.

    We continue to use the auxiliary function $g$ constructed in the proof of \cref{Prop-C^0-halfspace}, and set $h\coloneqq f-g$, $\tilde P\coloneqq\vert\bar\nabla h\vert^2$. A direct computation shows that $\bar\nabla_{\bar N}h=0$ on $T$, and
    \begin{align*}
        \begin{cases}
            \bar\De\tilde P\geq0\quad&\text{in }\Omega,\\
            \bar\nabla_{\bar N}\tilde P=0\quad&\text{on }T.
        \end{cases}
    \end{align*}
    It follows from the Hopf's boundary point lemma that $\tilde P$ attains its maximum on $\S$, and hence for any $y\in\Omega$, we have
    \begin{align*}
        &2\|\bar\nabla f\|_{\infty,\S}^2+2\|\bar\nabla g\|_{\infty,\S}^2
        \geq \vert\bar\nabla f(y)-\bar\nabla g(y)\vert^2\\
        =&\vert\bar\nabla f(y)\vert^2+\vert\bar\nabla g(y)\vert^2-2\left<\bar\nabla f(y),\bar\nabla g(y)\right>
        \geq\frac{\vert\bar\nabla f(y)\vert^2}{2}-\vert\bar\nabla g(y)\vert^2.
    \end{align*}
    Notice that $\|\bar\nabla g\|_{\infty,\Omega}$ is controlled by some positive constant $C$ depends only on $n,c,d_\Omega$, the assertion then follows easily.
\end{proof}

\subsection{Proof of the quantitative stability results}\label{Sec-4-3}
We begin with the following (geometric) lower bound estimate of $-f$, which is adapted from \cite[Lemma 4.1]{MP22}.
\begin{proposition}
Let $f$ be the solution of \eqref{eq-mixed-halfspace0}  for some $c\leq0$, then there holds
\begin{align}\label{ineq-MP22-4.1}
    -f(x)\geq \frac{1}{2(n+1)} \delta^{2}_{\partial \Omega}(x), \quad\forall x\in \bar{\Omega}.
\end{align}  
\end{proposition}
\begin{proof}
Since $c\leq0$, applying the maximum principle \cite[Lemma 4.1]{Pog22}, we see that $f$ is non-positive in $\O$.
For any $x\in\O$, we set $R=\delta_{\p\O}(x)$ and we consider the open ball $B_R(x)\subseteq\O$, notice that the function $w(y)=\frac{\vert y-x\vert^2-R^2}{2(n+1)}$ is the torsion potential, i.e.,
\begin{align*}
    \bar\De w=1\text{ in }B_R(x),\quad w=0\text{ on }\p B_R(x),
\end{align*}
by comparison we see that $w\geq f$ on $\overline B_R(x)$ and it follows that $-\frac{R^2}{2(n+1)}=w(x)\geq f(x)$, which proves \eqref{ineq-MP22-4.1}.
\end{proof}


\begin{proof}[Proof of \cref{Thm-stabi-OverDetermined-halfspace}]
For $n\geq2$,
our starting point is the estimate \eqref{ineq-rhoe-rhoi} with $\alpha=1$.
To finish the proof, it suffice to estimate the integral on the right. Recall that as shown in the proof of \cref{Thm-rigid-BVP-halfspace}, if $f_\nu\equiv$ constant on $\S$, then $f_\nu=\frac{R(c,\O)}{n+1}$, therefore by choosing $R=R(c,\O)$ in the integral identity \eqref{eq-integral-identity-1}, we get
\begin{align*}
    \int_\O(-f)\bar\De P\rd x\leq C(n,c,d_\O,L)\left\| f_\nu^2-(\frac{R(c,\O)}{n+1})^2\right\|_{1,\Sigma},
\end{align*}
which together with \eqref{ineq-MP22-4.1} yields that
\begin{align*}
    \int_\O\delta_{\p\O}^2\bar\De P\rd x\leq C(n,c,d_\O,L)\left\| f_\nu^2-(\frac{R(c,\O)}{n+1})^2\right\|_{1,\Sigma}.
\end{align*}
 Thus we obtain from \eqref{ineq-rhoe-rhoi} and \cref{Prop-GE-Halfspace} that
 \begin{align*}
     \rho_e-\rho_i\leq C(n,c,d_\O,\eta, a, \theta, r_e)\left\| f_\nu^2-(\frac{R(c,\O)}{n+1})^2\right\|^{\frac{1}{n+1}}_{1,\S}.
 \end{align*}

The proof for the case $n=1$ follows similarly by using \eqref{ineq-rhoe-rhoi-n=1}, which completes the proof.

\end{proof}

\begin{proof}[Proof of \cref{Thm-stabi-Meancurvature-halfspace}]
As verified in \cite[Proof of Theorem 1.1]{JXZ22}, for $\theta\in(0,\pi/2]$, there exists $f\in W^{2,2}(\Omega)\cap   W^{2,1}(T)\cap C^2(\overline{\Omega}\setminus \Gamma)\cap C^1(\overline{\Omega})$ that solves \eqref{eq-mixed-halfspace0} with $c=c(\theta,\O)=-\frac{n}{n+1}\cot\theta\frac{  |T|}{|\Gamma|}$.

For $n\geq2$, we consider \eqref{ineq-rhoe-rhoi} with $\alpha=0$.
To finish the proof, it suffice to estimate the integral on the right.
Recall that as shown in \cref{Prop-H1}, if $\S$ is of constant mean curvature, then $H=H(\theta,\O)$.
Observe also that the integral identity \eqref{iden-integral-meancurvature-1} implies
\begin{align*}
    \int_\O\bar\De P\rd x\leq L^2\left\| H(\theta,\O)-H\right\|_{1,\S},
\end{align*}
and it follows from \eqref{ineq-rhoe-rhoi} and \cref{Prop-GE-Halfspace} that
\begin{align*}
    \rho_e-\rho_i\leq C(n,\abs{T},\abs{\Gamma},d_\O,\eta, a, \theta, r_e)\left\| H(\theta,\O)-H\right\|_{1,\S}^{\frac{1}{n+1}}.
\end{align*}

The proof for the case $n=1$ follows similarly by using \eqref{ineq-rhoe-rhoi-n=1}, which completes the proof.
\end{proof}

\subsection{Quantitative stability of Heintze-Karcher-type inequalities}\label{Sec-4-4}
In this subsection, we provide an alternative approach to the quantitative stability estimate of the Alexandrov-type theorem through the well-known Heintze-Karcher inequality.

First we record the capillary Heintze-Karcher-type inequality in the half-space (\cite[Theorem 1.1]{JXZ22}), which states that for a domain $\Omega\subset\mfR^{n+1}_+$ such that
$\S$ intersects $\p\mfR^{n+1}_+$ at a constant contact angle $\theta\in (0,\frac{\pi}{2}]$ and $H>0$ on $\Sigma$,
there holds that
\begin{align}\label{ineq-HK-halfspace}
    \int_\S\frac{1}{H}\rd A
    \geq\frac{n+1}{n}\vert\O\vert+\cot\theta\frac{\vert T\vert^2}{\vert\Gamma\vert}
\end{align}

A refined version of \eqref{ineq-HK-halfspace} is as follows, which can be used to establish a quantitative estimate:
\begin{proposition}
Let $\Omega\subset\mfR^{n+1}_+$ be such that
$\S$ intersects $\p\mfR^{n+1}_+$ at a constant contact angle $\theta\in (0,\frac{\pi}{2}]$ and $H>0$ on $\Sigma$. Then
there holds that
\begin{align}\label{iden-integral-HK-1}
    \int_\O\left\{\vert\bar\nabla^2f\vert^2-\frac{(\bar\De f)^2}{n+1}\right\}\rd x
    \leq\left(\frac{n}{n+1}\right)^2\left\{\int_\S\frac{1}{H}\rd A-\left(\frac{n+1}{n}\vert\O\vert+\cot\theta\frac{\vert T\vert^2}{\vert\Gamma\vert}\right)\right\}.
\end{align}
\end{proposition}
\begin{proof}
As verified in \cite{JXZ22}, for $\theta\in(0,\pi/2]$, there exists a solution $f\in W^{2,2}(\Omega)\cap   W^{2,1}(T)\cap C^2(\overline{\Omega}\setminus \Gamma)\cap C^1(\overline{\Omega})$ to \eqref{eq-mixed-halfspace0} with $c=c(\theta,\O)$.
Recalling the definition of $c(\theta,\O)$, a simple integration by parts yields that
\begin{align*}
    \vert\O\vert
    =\int_\O\bar\De f\rd x
    =\int_\S f_\nu\rd A-\frac{n}{n+1}\cot\theta\frac{\vert T\vert^2}{\vert\Gamma\vert},
\end{align*}
and hence the H\"older inequality gives
\begin{align*}
    \left(\vert\O\vert+\frac{n}{n+1}\cot\theta\frac{\vert T\vert^2}{\vert\Gamma\vert}\right)^2
    =\left(\int_\S f_\nu\rd A\right)^2
    \leq\left(\int_\S Hf_\nu^2\rd A\right)\cdot\left(\int_\S\frac{1}{H}\rd A\right),
\end{align*}
putting this back into \eqref{eq-LaplaceP}, we find
\begin{align*}
    \int_\O\bar\De P\rd x
    \leq&\frac{n}{n+1}\vert\O\vert+(\frac{n}{n+1})^2\cot\theta\frac{\vert T\vert^2}{\vert\Gamma\vert}-\frac{\left(\vert\O\vert+\frac{n}{n+1}\cot\theta\frac{\vert T\vert^2}{\vert\Gamma\vert}\right)^2}{\int_\S\frac{1}{H}\rd A}\\
    =&\frac{n}{n+1}\left(\vert\O\vert+\frac{n}{n+1}\cot\theta\frac{\vert T\vert^2}{\vert\Gamma\vert}\right)\left(\frac{\int_\S\frac{1}{H}\rd A-\frac{n+1}{n}\left(\vert\O\vert+\frac{n}{n+1}\cot\theta\frac{\vert T\vert^2}{\vert\Gamma\vert}\right)}{\int_\S\frac{1}{H}\rd A}\right)\\
    \leq&\left(\frac{n}{n+1}\right)^2\left\{\int_\S\frac{1}{H}\rd A-\frac{n+1}{n}\left(\vert\O\vert+\frac{n}{n+1}\cot\theta\frac{\vert T\vert^2}{\vert\Gamma\vert}\right)\right\},
\end{align*}
where we have used \eqref{ineq-HK-halfspace} for the last inequality. This completes the proof.
\end{proof}

\begin{theorem}\label{Thm-stabi-HK-halfspace}
Let $\O\subset\mathbb{R}_+^{n+1}$ satisfy the $(\eta,a)$-uniform interior cone condition and the $r_e$-uniform exterior $\theta$-spherical cap condition for some $\eta\in(0,\pi/2]$, $a>0$, $\theta\in(0,\pi/2], r_e>0$.
Assume that $\S=\overline{\p\O\cap \mfR^{n+1}_+}$ meets $\p\mfR^{n+1}_+$ at the constant contact angle $\theta\in(0,\pi/2]$ and $H>0$ on $\Sigma$.
Let $O$ be the center chosen as in \eqref{defn-O-halfspace},
and define
\begin{align*}
\epsilon\coloneqq
\int_\S\frac{1}{H}\rd A-\left(\frac{n+1}{n}\vert\O\vert+\cot\theta\frac{\vert T\vert^2}{\vert\Gamma\vert}\right)\geq0.
\end{align*}
Then there exists some constant $C=C(n,\abs{T},\abs{\Gamma},d_\O,\eta, a, \theta, r_e)>0$ such that
\begin{align*}
\rho_e-\rho_i
\leq C
\begin{cases}
\epsilon^{\frac{1}{n+1}},\quad&n\geq2,\\
\epsilon^{\frac{1}{2}}\max\left\{\log\left(\epsilon^{-\frac12}\right),1\right\},\quad&n=1.
\end{cases}
\end{align*}
\end{theorem}

\begin{proof}
For $n\geq2$, we consider \eqref{ineq-rhoe-rhoi} with $\alpha=0$.
Note that the integral on the right is controlled thanks to the integral inequality \eqref{iden-integral-HK-1}.
So the assertion follows easily by using \cref{Prop-GE-Halfspace}.

The proof for the case $n=1$ follows similarly by using \eqref{ineq-rhoe-rhoi-n=1}, which completes the proof.
\end{proof}

\

\noindent{\bf 
Statements and Declarations.}
Data sharing not applicable to this article as no datasets were generated or analysed during the current study.

\

\bibliographystyle{alpha}
\bibliography{Stability.bib}
\end{document}